\documentclass[a4paper,11pt]{amsart}
\usepackage[utf8x]{inputenc}
\usepackage{amsthm,amsfonts,amsmath,amssymb,enumerate}
\usepackage{verbatim} 
\usepackage{rotating} 

\usepackage{hyperref}
\usepackage{pdfsync}
\usepackage{multirow}

\usepackage{color}

\renewcommand{\L}{\mathcal{L}}
\newcommand{\N}{\mathcal{N}}
\newcommand{\F}{\mathcal{F}}

\newcommand{\C}{\mathbb{C}}
\newcommand{\g}{\mathfrak{g}}
\newcommand{\n}{\mathfrak{n}}
\newcommand{\h}{\mathfrak{h}}
\newcommand{\f}{\mathfrak{f}}

\newcommand{\ad}{\operatorname{ad}}

\newcommand {\Der} {\operatorname{Der}}

\newtheorem{theorem}{Theorem}[section]
\newtheorem{lemma}[theorem]{Lemma}

\newtheorem{proposition}[theorem]{Proposition}

\theoremstyle{definition}

\theoremstyle{remark}
\newtheorem{remark}[theorem]{Remark}

\numberwithin{equation}{section}

\begin{document}

\title[No rigid filiform Lie algebras]{There are no rigid filiform Lie algebras \\ of low dimension}
\author{Paulo Tirao and Sonia Vera}
\address{CIEM-FaMAF, Universidad Nacional de Córdoba, Argentina}
\date{August, 2017}

\maketitle

\begin{abstract}
We prove that there are no rigid complex filiform Lie algebras in the variety of (filiform) Lie algebras 
of dimension less than or equal to 11.
More precisely we show that in any Euclidean neighborhood of a filiform Lie bracket (of low dimension), 
there is a non-isomorphic filiform Lie bracket.
This follows by constructing non trivial linear deformations in a Zariski open dense set of the variety of filiform Lie
algebras of dimension 9, 10 and 11. (In lower dimensions this is well known.) 
\end{abstract}

\section{Introduction}

We are interest on the existence problem of rigid nilpotent Lie algebras of finite dimension,
mainly over the complex numbers.
In this case we consider a Lie algebra as rigid if its isomorphism class is open in the euclidean topology.
In general a Lie algebra is rigid if its isomorphism class is open in the Zariski topology.
The expected answer is no, there are no rigid nilpotent Lie algebras, and it is usually referred as Vergne's conjecture.

In this paper we address this problem for filiform Lie algebras.
Our approach and the general setting applies for filiform algebras of arbitrary dimension.
However, at the moment, in order to obtain complete results we use the complete decomposition
in irreducible components of the varieties of filiform Lie algebras, something available only for small dimensions.

We prove that filiform Lie algebras of low dimension are never rigid,
not only in the variety of all Lie algebras, but also in the (sub)variety of filiform Lie algebras.

The variety $\L$ of complex Lie algebras of dimension $n$
is the algebraic set of alternating bilinear maps $\mu:\C^n\times\C^n\rightarrow\C^n$
that satisfy the Jacobi identity.
The orbit of $\mu\in\L$ under the natural action of $GL_n$, $O(\mu)$, is the isomorphism class of $\mu$.
We say that $\mu$ is rigid if $O(\mu)$ is open in the euclidean topology.

Inside $\L$ there are various interesting subvarieties to consider, for example the variety
$\N$ of nilpotent Lie algebras.
It is expected that in any neighborhood of a nilpotent Lie algebra $\mu$,
there are non isomorphic
Lie algebras.
This is the case, for instance, if the nilpotent Lie algebra $\mu$ has a semisimple derivation
\cite{C,GH,GT}.

Inside $\N$, the subvarieties $\N_k$ of nilpotent Lie algebras of nilpotency class less than or equal to $k$
are one inside the other $\N_k\subseteq\N_{k+1}$. 
The class of filiform Lie algebras $\F$, of algebras of maximum nilpotency class equal to $n-1$, is the Zariski open
complement of $\N_{n-2}$, and hence is a subvariety which is Zariski dense 
in all irreducible components of $\N$ that intersect $\F$.
Hence $\F$ is dense also in the Euclidean topology in all these irreducible components, and therefore
if the filiform algebras are non rigid it follows that all algebras in these irreducible components are
non rigid too.

We propose a general construction of linear deformations for nilpotent Lie algebras, 
that carries out nicely for filiform Lie algebras of arbitrary dimension.
In order to show that these deformations are non trivial, the usual Lie algebra invariants are not sufficient.
Instead we show directly that an isomorphism between the deformed algebra and the original can not exist.
This requires to study carefully the isomorphism equations that arise with respect to \emph{adapted basis}
of the filiform algebras involved.

The description of all filiform Lie algebras using adapted basis goes back to Michelle Vergne \cite{V}.
For small dimensions this makes it possible to write down an explicit parametrization with not many parameters \cite{GJK2}.
The structure of the varieties of filiform Lie algebras of dimension $n=9,10,11$ is relevant for us,
in particular their decomposition into irreducible components \cite{K,GJK1}.

It is worth mentioning that for smaller dimensions, $n=3\dots 8$ complex filiform Lie algebras are classified, and it is already known that
there are no rigid filiform Lie algebras. 
However, in dimensions $\le 6$ there are only a finite number of (isomorphism classes of) them, and
there are algebras which are rigid inside the variety of filiform Lie algebras.

\section{Linear deformations of Lie algebras}

Given bilinear maps $\psi$ and $\phi$ let $\psi\circ\phi$ be the trilinear map defined by
\begin{eqnarray*}
 \psi\circ\phi(X,Y,Z) &=& \circlearrowleft \psi(\phi(X,Y),Z) \\
                      &=& \psi(\phi(X,Y),Z)+\psi(\phi(Y,Z),X)+\psi(\phi(Z,X),Y).
\end{eqnarray*}
The Jacobi identity for $\mu$ is $\mu\circ\mu=0$.
Hence an alternating bilinear map $\mu$ is a Lie algebra if and only if $\mu\circ\mu=0$.
In addition, a bilinear map $\phi$ is a 2-cocycle for $\mu$ if and only if $\mu\circ\phi+\phi\circ\mu=0$.

For the purpose of this paper, given a Lie algebra $\mu$, we shall consider
linear deformations of $\mu$, that is
families of Lie algebras $\mu_t=\mu+t\varphi$, where $t$ is a complex parameter.
It is straightforward to verify that $\mu_t$ is a Lie algebra for all $t$ if and only if $\varphi$ is a Lie algebra and a 2-cocycle for $\mu$.

If for arbitrary small $\epsilon$, there is a $t\in\C$ with $|t|<\epsilon$ such that $\mu_t$ is not isomorphic to $\mu$, then $\mu$ is not rigid.
In case $\mu_t$ is a curve of nilpotent Lie algebras (or filiform), 
then $\mu$ is not rigid in the variety of nilpotent (or filiform) Lie algebras.

In this section we propose a method for the construction of linear deformations of Lie algebras
that we apply to filiform Lie algebras in Section \ref{sec:deformations}.

In order to refer to the commutator or an ideal of a given Lie algebra $\mu$,
we need to mention explicitly the underlying vector space of $\mu$.

\begin{proposition}\label{prop:deformation}
Let $\mu$ be a nilpotent Lie algebra defined on $\n=\C^n$ and let $\h$ be an ideal of $\n$ of codimension 2 containing the commutator $[\n,\n]$.
Let $\langle x_0,x_1\rangle$ be a direct complement to $\h$, so that $\n=\langle x_0,x_1\rangle \oplus \h$.

If $D\in\Der(\h)$ is such that $D\ad(x_0)=\ad(x_0) D$ in $\h$, then the alternating bilinear map $\varphi_D$ defined by
\[ 
 \varphi_D(x_0,x_1)=0, \quad \varphi_D(x_0,h)=0, \quad \varphi_D(x_1,h)=D(h), \quad \varphi_D(h,h')=0,
\]
for $h,h'\in\h$,
is a Lie algebra and a 2-cocycle for $\mu$.
Therefore $\mu_t=\mu+t\varphi_D$ is a linear deformation of $\mu$.
\end{proposition}

\begin{proof}
We first check that $\circlearrowleft \varphi_D(\varphi_D(X, Y), Z)=0$ for $X=x_0$, $Y=x_1$ and $Z=h\in\h$;
we have that
\[
\circlearrowleft \varphi_D(\varphi_D(x_0, x_1), h)= 
  \varphi_D (0, h) + \varphi_D(D(h), x_0) + \varphi_D(0, x_1) = 0.
\]
All other cases are clear since $\varphi_D(h_1,h_2)=0$ for $h_1,h_2\in\h$.


To show that $\mu\circ\varphi_D+\varphi_D\circ\mu=0$ we check that
$\circlearrowleft \mu(\varphi_D(X, Y), Z) + \circlearrowleft \varphi_D(\mu (X, Y), Z)=0$
in three cases: for $X=x_0$, $Y=x_1$ and $Z\in\h$; for $X=x_0$ and $Y,Z\in\h$; or for $X=x_1$ and $Y,Z\in\h$.
There is no need to check the case in which $X,Y,Z\in\h$, since it is clear.

In the first case we have that:
\begin{eqnarray*}
\circlearrowleft \mu(\varphi_D(x_0, x_1), h) + \circlearrowleft \varphi_D(\mu (x_0, x_1), h) \\
  &&\kern-5cm = \mu(0, h) + \mu(D (h), x_0) + \mu(0, x_1) + 0 + 0 - \varphi_D(x_1, \mu (h, x_0)) \\
  &&\kern-5cm = -\mu(x_0, D(h)) + D(\mu (x_0, h)) \\
  &&\kern-5cm =  0,
\end{eqnarray*}
since $D\ad(x_0)=\ad(x_0) D$.
Notice that $\varphi_D(\mu (x_0, x_1), h)=0$ because $\mu(x_0,x_1)\in\h\oplus\langle x_0 \rangle$,
otherwise $\mu$ would not be nilpotent.

In the second case we have that:
\begin{eqnarray*}
\circlearrowleft \mu(\varphi_D(x_0, h), h') + \circlearrowleft \varphi_D(\mu (x_0, h), h') \\
  &&\kern-1.5cm = \mu(0, h') + \mu(0, x_0) + \mu(0, h) + 0 + 0 + 0 \\
  &&\kern-1.5cm =  0.
\end{eqnarray*}

In the third case we have that:
\begin{eqnarray*}
\circlearrowleft \mu(\varphi_D(x_1, h), h') + \circlearrowleft \varphi_D(\mu (x_1, h), h') \\
  &&\kern-5cm = \mu(D(h), h') + \mu(0, x_1) - \mu(D(h'), h) + 0 - D(\mu(h,h'))  + 0 \\
  &&\kern-5cm =  0,
\end{eqnarray*}
since $D$ is a derivation of $\h$.
\end{proof}

\begin{remark}
 The construction above applies not only to nilpotent Lie algebras. 
 In fact, the nilpotency hypothesis was used in the proof only to ensure
 that $\mu(x_0,x_1)\in\h\oplus\langle x_0 \rangle$.
 
 For instance, the construction applies to the following easy example.
 Let $\g=\langle x_0,x_1 \rangle \oplus \C^{n-2}$, with bracket $\mu$, be the direct sum of the 
 2-dimensional solvable Lie algebra given by $\mu(x_0,x_1)=x_0$
 and the abelian Lie algebra of dimension $n-2$.
 Choose $\h=\C^{n-2}$. Since $\mu(x_0,x_1)=x_0$, the same proof carries over.
 Therefore we may choose $D$ a derivation of $\h$, that is, an endomorphisms of $\h$.
 Since $D$ commutes with $\ad_{x_0}=0$ in $\h$, we may choose any $D$.
 Hence, the corresponding $\varphi_D$ gives rise to the linear deformation
 $\mu_t=\mu+t\varphi_D$ of $\mu$ given by:
 \[ \mu_t(x_0,x_1)=x_0,\quad \mu_t(x_0,h)=0,\quad \mu_t(x_1,h)=tD(h),\quad \mu_t(h,h')=0. \]
\end{remark}

\begin{remark}[The Grunewald-O'Halloran method]\label{rmk:gh}
Our construction was inspired by a previous one by Grunewald and O'Halloran \cite{GH}.
Given a Lie algebra $\g$ with bracket $\mu$, an ideal $\h$ of codimension 1, a derivation $D$ of $\h$
and $x\not\in \h$,
the alternating bilinear map $\varphi_D$ defined by
\[ 
 \varphi_D(x,h)=D(h), \quad \varphi_D(h,h')=0,
\]
for $h,h'\in\h$,
is a Lie algebra and a 2-cocycle for $\mu$.
Therefore $\mu_t=\mu+t\varphi_D$ is a linear deformation of $\mu$.

As a relevant application they showed that if $\g$ is a nilpotent Lie algebra and $D$ is a semisimple 
(or has a non-zero eigenvalue) derivation of a given ideal $\h$ of codimension 1,
then the corresponding linear deformation $\mu_t$ is non-trivial (solvable non-nilpotent).

As a consequence, since every nilpotent Lie algebra of dimension $\le 6$ admits a semisimple derivation,
any nilpotent Lie algebra of dimension $\le 7$ has a solvable non-nilpotent deformation.
In particular there are no rigid nilpotent Lie algebras of dimension $\le 7$.
\end{remark}

\section{Filiform Lie algebras}

The following description of all filiform Lie algebras may be found in \cite{GJK2} and goes back to \cite{V}.
Given a filiform Lie algebra $\mu$ of dimension $n$, there is an {\sl adapted basis} $\{x_0,x_1,\dots,x_{n-1}\}$ 
of $\C^{n}$ such that
\[ \mu=\mu_0+\psi, \]
where $\mu_0$ is the standard filiform Lie algebra given by:
\[ \mu_0(x_i,x_j)= \begin{cases} x_{j+1}, &\text{if $i=0$,} \\
                               0,  &\text{if $i\ne 0$,}
                 \end{cases}
\]
and $\psi$ is a 2-cocycle for $\mu$ that may be written as 
\begin{equation}\label{eqn:psi}
\psi = \sum_{(r,s)\in\Delta_{n}} a_{r,s} \psi_{r,s}
\end{equation}
for some particular cocycles $\psi_{r,s}$ and $a_{r,s}\in\C$. 
The cocycle $\psi_{r,s}$ is given by:
\[
\psi_{r,s} (x_i, x_j) = -\psi_{r,s} (x_j, x_i) = (-1)^{r-i} \binom{j - r - 1}{r - i} x_{i+j+s-2r-1},
\]
for $1\le i\le r<j\le n-1$ and $i+j+s-2r-1\le n-1$, and $\psi_{r,s}(x_i,x_j)=0$ otherwise.
The index set $\Delta_n$ is, depending on whether $n$ is odd or even respectively, as follows:
\begin{eqnarray*}
 \Delta_{n} &=& \big\{ (r,s):\, 1\le r\le n-2,\ 2r+1<s\le n-1\big\} \\ 
 \Delta_{n} &=& \big\{ (r,s):\, 1\le r\le n-2,\ 2r+1<s\le n-1\big\} \cup \big\{(\tfrac{n-2}{2},n-1) \big\}. 
\end{eqnarray*}

We collect some facts, that follow directly from this description of $\mu$, for later use.

\begin{lemma}\label{lemma:mu}
 Given $\mu$ of dimension $n$ and $\{x_0,\dots,x_{n-1}\}$ an adapted basis for $\mu$ 
 it holds that:
 \begin{enumerate}[(1)]
  \item $\mu(x_0,x_i)=x_{i+1}$, for $1\le i\le n-2$.
  
  \item $\mu(x_i,x_j)=0$, for $\frac{n-2}{2}<i<j$.
  
  \item If $n$ is odd, $\mu(x_i,x_j)=\sum_{k\ge i+j+1} c_{i,j}^k x_k$, for $i<j$.
  
  \item If $n$ is even, $\mu(x_i,x_j)=\sum_{k\ge i+j} d_{i,j}^k x_k$, for $i<j$,
    where $d_{i,j}^{i+j}$ might be different from 0 only for $1\le i\le \frac{n-2}{2}<j\le n-1$.
    
  \item The central descending series of $\mu$ is given by: 
\begin{eqnarray*} 
C_0 &=& \langle x_0,x_1,x_2,\dots,x_{i+1},\dots,x_{n-1} \rangle \\
C_1 &=& \langle x_2,\dots,x_{i+1},\dots,x_{n-1} \rangle \\
    &\vdots& \\
C_i &=& \langle x_{i+1},\dots,x_{n-1} \rangle \\
    &\vdots& \\
C_{n-2} &=& \langle x_{n-1} \rangle 
\end{eqnarray*}
 
 \end{enumerate}
\end{lemma}

\subsection{The varieties $\F^9$, $\F^{10}$ and $\F^{11}$} \label{subsec:F}

\ 

According to the description above, the varieties $\F^n$ of filiform Lie algebras of dimension $n$ can be parametrized
by the algebraic set of parameters $a_{r,s}\in\Delta_n$ satisfying the polynomial equations
\[ \psi\circ\psi=0, \]
where $\psi$ is as in \eqref{eqn:psi}.

For small $n$, as $n=9,10,11$, the polynomial equations describing the corresponding varieties
$\F^{9}$, $\F^{10}$ and $\F^{11}$ are not difficult to compute.
The irreducible components are much more difficult to obtain.
Their are described in \cite{K} and were computed originally in \cite{GJK1}.

The variety of complex filiform Lie algebras of dimension 9, $\F^9$, may be described   
as the set of complex parameters $a_{r,s}\in\Delta_9$ satisfying 
\begin{equation}\label{eqn:dim9}
 - 3 a_{2,6}^2 + a_{2,6} a_{3,8} + 2 a_{1,4} a_{3,8} = 0.
\end{equation}
It turns out that $\F^9$ is irreducible.

Similarly, $\F^{10}$ may be described as the set of parameters in $\Delta_{10}$ satisfying the equations
\begin{equation}\label{eqn:dim10}
  \begin{aligned}
- 3 a_{2,6}^2 + a_{2,6} a_{3,8} + 2 a_{1,4} a_{3,8} &= 0 \\ 
- 7 a_{2,6} a_{2,7} + (2 a_{1,4} + a_{2,6}) a_{3,9} + 3 ( a_{1,5} + a_{2,7} ) a_{3,8} \qquad \\
                                              - ( a_{2,8} + 2 a_{1,6}) a_{4,9}  &= 0 \\ 
a_{4,9} ( 2 a_{1,4} - a_{2,6} - a_{3,8}) &= 0
  \end{aligned}
\end{equation}
However, $\F^{10}$ is not longer irreducible and has three components, given by 
\[
 C^{10}_1 : \left\{ 
          \begin{aligned} 
          a_{4,9} &= 0\\
          - 3 a_{2,6}^2 + a_{2,6} a_{3,8} + 2 a_{1,4} a_{3,8} &= 0 \\ 
          3 a_{2,6}^2 a_{3,9} + (-7 a_{2,6} a_{2,7} + 3 a_{1,5} a_{3,8} + 3 a_{2,7} a_{3,8}) a_{3,8} &= 0 \\ 
          -7 a_{2,6} a_{2,7} +( 3 a_{1,5} + 3 a_{2,7} ) a_{3,8} + (2 a_{1,4} + a_{2,6} ) a_{3,9} &= 0\\
           \end{aligned}
        \right. 
\] 
\[
C^{10}_2 : \left\{ 
          \begin{aligned}  
          a_{2,6} - a_{3,8} &= 0 \\
          ( 3 a_{1,5} - 4 a_{2,7} + 3 a_{3,9} ) a_{3,8} - ( 2 a_{1,6} + a_{2,8} ) a_{4,9} &= 0\\
          a_{1,4} - a_{3,8} &= 0 
           \end{aligned}
        \right.
\]
\[
C^{10}_3 : \left\{ 
          \begin{aligned}
          3 a_{2,6} + a_{3,8} &= 0 \\
         (9 a_{1,5} + 16 a_{2,7} + a_{3,9} ) a_{3,8} - 3 (2 a_{1,6} +  a_{2,8} ) a_{4,9} &= 0 \\
         3 a_{1,4} - a_{3,8} &= 0
          \end{aligned}
        \right.
\] 

Finally, $\F^{11}$ can be described as the set of parameters in $\Delta_{11}$ satisfying the equations
\begin{equation}\label{eqn:dim11}
 \begin{aligned}
- 3 a_{2,6}^2 + a_{2,6} a_{3,8} + 2 a_{1,4} a_{3,8} &= 0 \\ 
(6 a_{3,8} - 4 a_{2,6}) a_{3,8} + (2 a_{1,4} - a_{2,6} - a_{3,8}) a_{4,10} &= 0\\
- 7 a_{2,6} a_{2,7} + (2 a_{1,4} + a_{2,6}) a_{3,9} + 3 ( a_{1,5} + a_{2,7} ) a_{3,8} &= 0 \\ 
- 4 a_{2,7}^2 - 8 a_{2,6} a_{2,8} + (4 a_{1,6} + 6 a_{2,8}) a_{3,8} + 3(a_{1,5} + a_{2,7}) a_{3,9}\\
 + (2 a_{1,4} + a_{2,6}) a_{3,10} - (2 a_{1,6} + a_{2,8}) a_{4,10} &= 0
\end{aligned}
\end{equation}
It is also not irreducible and has two irreducible components given by
\[
 C^{11}_1 : \left\{ 
          \begin{aligned} 
          a_{1,4} - a_{2,7} &= 0\\
          - 3 a_{2,6}^2 + a_{2,6} a_{3,8} + 2 a_{1,4}a_{3,8} &= 0 \\ 
          - 2 (2 a_{2,6} - 3 a_{3,8}) a_{3,8} + (2a_{2,7} - a_{2,6} - a_{3,8} ) a_{4,10} &= 0 \\
         - 7 a_{2,6}a_{2,7} + (2a_{2,7} + a_{2,6}) a_{3,9} + (3 a_{1,5}+3a_{2,7})a_{3,8}  &= 0 \\ 
         - 4a_{2,7}^2 - (8 a_{2,6} -6 a_{3,8}+a_{4,10}) a_{2,8}  + (3a_{1,5}+3a_{2,7}) a_{3,9} \quad \\
            + (2a_{2,7} + a_{2,6}) a_{3,10} + (4a_{3,8}-2a_{4,10})a_{1,6} &= 0
           \end{aligned}
        \right. 
\]
\[
 C^{11}_2 : \left\{ 
          \begin{aligned} 
         a_{2,8}+a_{2,7}^2-3a_{2,7}a_{2,9}-3a_{2,9} &= 0 \\
         - 3 a_{2,6}^2 + a_{2,6} a_{3,8} + 2 a_{1,4}a_{3,8} &= 0 \\ 
         - 2 (2 a_{2,6} - 3 a_{3,8}) a_{3,8} + (2a_{1,4} - a_{2,6} - a_{3,8} ) a_{4,10} &= 0 \\
         -7a_{2,6}a_{2,7}+3(a_{1,5}+a_{2,7})a_{3,8}+(2a_{1,4}+a_{2,6})a_{3,9} &= 0 \\
         (-4+8a_{2,6}-6a_{3,8}+a_{4,10})a_{2,7}^2 \quad \\
         +(-24a_{2,6}a_{2,9}+18a_{3,8}a_{2,9}+3a_{3,9}-3a_{2,9}a_{4,10})a_{2,7} \quad \\
         +(-24a_{2,6}+3a_{4,10}+18a_{3,8})a_{2,9} + (4a_{3,8}-2a_{4,10})a_{1,6} \quad \\
         + (2a_{1,4}+a_{2,6})a_{3,10} + 3a_{1,5}a_{3,9} &= 0
           \end{aligned}
        \right. 
\]

\begin{remark}\label{rmk:78}
 In dimensions $\le 6$, there are only a finite number of (isomorphisms classes of) filiform Lie algebras,
 actually this is the case for all nilpotent Lie algebras.
 More precisely, in dimensions 3 and 4 there is only one filiform, the standard; 
 there are two of dimension 5 and three of dimension 6.
 In dimensions 7 and 8 there are already infinitely many isomorphisms classes of filiform Lie algebras.
 Moreover, $\F^7$ is the (irreducible) affine space given by $\Delta_7$, 
 and $\F^8$ is parametrized by $\Delta_8$ with the condition
 \[ a_{3,7}(2a_{1,4}+a_{2,6})=0. \]
 Its irreducible components are given by
 \[ C^8_1: \; a_{3,7}=0 \quad\text{and}\quad C^8_2: \; 2a_{1,4}+a_{2,6}=0. \] 
\end{remark}

\section{Linear deformations of filiform Lie algebras}\label{sec:deformations}

Given a filiform Lie algebra $\mu$, defined on $\f=\C^n$ with $n\ge 6$, we use Proposition \ref{prop:deformation}
to construct a linear deformation of it.
To this end consider the ideal $\h=[\f,\f]$ and fix and adapted basis $\{x_0,x_1,\dots,x_{n-1}\}$ of $\f$.
Notice that $\h=\langle x_2,\dots,x_{n-1} \rangle$.

\subsection{Two families of deformations}

We define two different derivations of the ideal $\h$, 
one of it only for the case $n$ is odd, that satisfy the hypothesis of
Proposition \ref{prop:deformation}, and then give rise to two families of
deformations.

\begin{proposition}
 The linear transformations $D^3:\h\rightarrow \h$, defined by
 \[
  D^3(x_2)=x_{n-3},\quad D^3(x_3)=x_{n-2},\quad D^3(x_4)=x_{n-1},
 \]
 and $D^3(x_i)=0$ for $i\ge 5$,
 is a nilpotent derivation of $\h$ and $D^3 \ad(x_0)=\ad(x_0) D^3$ in $\h$.
 
 If $n$ is odd, then the linear transformation $D^4:\h\rightarrow \h$, defined by
 \[ 
  D^4(x_2)=x_{n-4},\quad D^4(x_3)=x_{n-3},\quad D^4(x_4)=x_{n-2},\quad D^4(x_5)=x_{n-1},        
 \]
 and $D^4(x_i)=0$ for $i\ge 6$,
 is a nilpotent derivation of $\h$ and $D^4 \ad(x_0)=\ad(x_0) D^4$ in $\h$.
\end{proposition}

\begin{proof}
Let $2\le i<j$.
By Lemma \ref{lemma:mu}, we have that
\[ \mu(x_i, x_j) = \sum_{k\ge 5} c_{ij}^k x_{k}, \]
and if $n$ is odd, actually
\[ \mu(x_i, x_j) = \sum_{k\ge 6} c_{ij}^k x_{k}. \]
Hence, $D^3(\mu(x_i,x_j))=0$ and $D^4(\mu(x_i,x_j))=0$, for $x_i,x_j\in\h$.

If $i\ge 5$, then $\mu(D^3(x_i), x_j))=0$ and also $\mu(x_i, D^3(x_j))=0$.
For $i=2,3,4$ ($j\ge 3$), we have that
\[
\mu(D^3(x_i), x_{j})) = \mu(x_{n-5+i},x_{j}) = 0
\]
because $n-5+i+j\ge n$ (see Lemma \ref{lemma:mu}).
And we also have that
\[
\mu(x_i, D^3(x_j)) = \mu(x_i,\lambda x_{n-5+j}) = 0,
\]
as before, where $\lambda=1$ if $j=3,4$ and $\lambda=0$ otherwise.
Therefore, $D^3$ is a derivation of $\h$.
It is clear that $D^3$ is nilpotent, since $n\ge 6$.

Assume now $n$ is odd.
If $i\ge 6$, then $\mu(D^4(x_i), x_j))=0$ and also $\mu(x_i, D^4(x_j))=0$.
For $i=2,3,4,5$ ($j\ge 3$), we have that
\[
\mu(D^4(x_i),x_{j})) = \mu(x_{n-6+i},x_{j}) = 0
\]
because $n-6+i+j+1\ge n$ (see Lemma \ref{lemma:mu}).
And we also have that
\[
\mu(x_i, D^4(x_j)) = \mu(x_i,\lambda x_{n-6+j}) = 0,
\]
as before, where $\lambda=1$ if $j=3,4,5$ and $\lambda=0$ otherwise.
Therefore, $D^4$ is a derivation of $\h$.
It is clear that $D^4$ is nilpotent, since $n\ge 7$ ($n\ge 6$ and odd).

That $D^3 \ad(x_0)=\ad(x_0) D^3$ in $\h$, follows by taking
$i\ge 2$ and evaluating
\begin{equation*}
D^3( \ad(x_0) x_i) = D^3(x_{i+1}) = 
  \begin{cases}
   x_{n-2} & \text{if $i = 2$} \\
   x_{n-1} & \text{if $i = 3$} \\
      0    & \text{if $i\ge 4$}
  \end{cases}
\end{equation*}
and 
\begin{equation*}
\ad(x_0) D^3(x_i) =
 \begin{cases}
  \ad(x_0)x_{n-3} = x_{n-2}, & \text{if $i=2$} \\
  \ad(x_0)x_{n-2} = x_{n-1}, & \text{if $i=3$} \\
  \ad(x_0)x_{n-1} = 0,       & \text{if $i=4$} \\
         0,                  & \text{if $i\ge 5$}
 \end{cases}
\end{equation*}

Similarly, that $D^4 \ad(x_0)=\ad(x_0) D^4$ in $\h$, follows by taking
$i\ge 2$ and evaluating
\begin{equation*}
D^4( \ad(x_0) x_i) = D^4(x_{i+1}) = 
  \begin{cases}
   x_{n-3} & \text{if $i=2$} \\
   x_{n-2} & \text{if $i=3$} \\
   x_{n-1} & \text{if $i=4$} \\
      0    & \text{if $i\ge 5$}
  \end{cases}
\end{equation*}
and 
\begin{equation*}
\ad(x_0) D^4(x_i) =
 \begin{cases}
  \ad(x_0)x_{n-4} = x_{n-3}, & \text{if $i=2$} \\
  \ad(x_0)x_{n-3} = x_{n-2}, & \text{if $i=3$} \\
  \ad(x_0)x_{n-2} = x_{n-1}, & \text{if $i=4$} \\
  \ad(x_0)x_{n-1} = 0,       & \text{if $i=5$} \\
         0,                  & \text{if $i\ge 6$}
 \end{cases}
\end{equation*}
\end{proof}

In what follows we shall consider, for each $n\ge 9$, one of the linear deformations of $\mu$ given by
\begin{equation}\label{eqn:mut}
 \mu_t=\mu+t\varphi_D,
\end{equation}
where $D=D^3$ or $D=D^4$. 
Precisely, if $n$ is even, we choose $D=D^3$ and if $n$ is odd we choose $D=D^4$,
except for $n=9$ in which case to choose $D=D^3$.

\begin{remark}\label{rmk:phiD}
It turns out, and is easy to check, that $\varphi_{D^3}=\psi_{1,n-3}$ for all $n$
and that $\varphi_{D^4}=\psi_{1,n-4}$ if $n$ is odd.
In particular, an adapted basis for $\mu$ is also an adapted basis for $\mu_t$.
\end{remark}

\subsection{On the isomorphisms between $\mu_t$ and $\mu$}\label{subsec:isos}

From now on let $n\ge 9$, $\mu\in\F^n$ and $\mu_t$ be the linear deformation of $\mu$ 
given in $\eqref{eqn:mut}$.

\begin{remark}
 The choices in $\eqref{eqn:mut}$ were made, in part, to simplify the proofs.
 On the one hand, the choice for $n=9$ allows us to include this case in Lemma \ref{lemma:m11} with all other cases,
 and on the other hand the choice for $n=11$ is crucial for us, since we are not able to prove that the deformation 
 obtained with the other choice is non-trivial.
\end{remark}

Let $g$ be an isomorphism between $\mu_t$ and $\mu$, and $[g]$
its matrix with respect to an adapted basis $\{x_0,\dots,x_{n-1}\}$ for $\mu$ and $\mu_t$.

Since $g$ preserves the central descending series (see Lemma \ref{lemma:mu}), then
\begin{equation}\label{eqn:g}
[g] = \begin{pmatrix}
 m_{1,1} & m_{1,2} & 0 & \cdots & 0\\
 m_{2,1} & m_{2,2}  &  0  & \cdots & 0\\
 m_{3,1} & m_{3,2} & m_{3,3}  & \cdots & 0\\
 \vdots & \vdots & \vdots  & \ddots & \vdots\\
 m_{n,1} & m_{n,2} & m_{n,3} & \cdots & m_{n,n}
 \end{pmatrix}
\end{equation}
Notice that
\[ gx_j=\sum_{i=1}^n m_{i,j+1}x_{i-1}, \qquad\text{for all $j=0,\dots,n-1$.} \]

That $g$ is an isomorphism is equivalent to 
\[ E_{i,j}=g\mu_t(x_i,x_j)-\mu(gx_i,gx_j)=0 \]
for all $0\le i<j\le n-1$.
If $E_{i,j}^k$ is the coefficient of $x_k$ in $E_{i,j}$, 
that is \[ E_{i,j} = g\mu_t(x_i,x_j)-\mu(gx_i,gx_j)=\sum_k E_{i,j}^k x_k, \]
then $g$ is an isomorphism if and only if $E_{i,j}^k=0$, for all $0\le i<j\le n-1$ and $0\le k\le n-1$.

\begin{remark}
 Since $\varphi_D(x_i,x_j)=0$ for almost all pairs $(i,j)$, whatever $D=D^3$ or $D=D^4$,
 most of the equations $E_{i,j}^k=0$ do not depend on the choice of $D$. 
 In fact $\varphi_D(x_i,x_j)\ne 0$ only for $i=1$ and $j=2,3,4$, if $D=D^3$ 
 and $\varphi_D(x_i,x_j)\ne 0$ only for $i=1$ and $j=2,3,4,5$, if $D=D^4$.
 Hence, $E_{i,j}^k=0$ depend on $D$ only for these pairs $(i,j)$ and some values of $k$.
 All other equations are in fact automorphisms equations for $\mu$ itself.
\end{remark}

The goal is to prove that such an isomorphism $g$ only can exist if $t=0$.
We start by showing that the matrix above is always lower triangular and
generically its diagonal entries are all equal to 1.

\begin{lemma}\label{lemma:m12}
 In the matrix $[g]$ in \eqref{eqn:g}, $m_{1,2}=0$.
\end{lemma}

\begin{proof}
Consider the equation $E_{1,n-3}^{n-2}=0$.
Since
\begin{eqnarray*}
 gx_1 &=& m_{1,2}x_0+m_{2,2}x_1+\dots+m_{n,n}x_{n-1} \\
 gx_{n-3} &=& m_{n-2,n-2}x_{n-3}+m_{n-1,n-1}x_{n-2}+m_{n,n}x_{n-1},
\end{eqnarray*}
then
\begin{eqnarray*}
 E_{1,n-3} &=& g\mu_t(x_1,x_{n-3})-\mu(gx_1,gx_{n-3}) \\
           &=& \alpha_1 x_{n-1} -m_{1,2}m_{n-2,n-2}x_{n-2}-\alpha_2 x_{n-1},           
\end{eqnarray*}
for some $\alpha_1,\alpha_2\in\C$.
Therefore, $E_{1,n-3}^{n-2}=0$ is equivalent to $m_{1,2}m_{n-2,n-2}=0$, and since $m_{n-2,n-2}\ne 0$,
it follows that $m_{1,2}=0$.
\end{proof}

\begin{remark}
The argument is even more direct if we assume $n$ is odd.
In this case, it is easier to evaluate 
$E_{1,n-2}=-m_{1,2}m_{n-1,n-1}x_{n-1}$ from were it follows, as before, that $m_{1,2}=0$. 

In what follows we shall try to keep the arguments independent from the parity of $n$.
\end{remark}

\begin{lemma}\label{lemma:mii}
 The $i$-th diagonal entry of $[g]$, for $3\le i\le n-1$ is $m_{i,i}=m_{2,2}m_{1,1}^{i-2}$. 
 Hence, the diagonal entries of $[g]$ are
 \[ m_{1,1},\ m_{2,2},\ m_{1,1}m_{2,2},\ m_{1,1}^2m_{2,2},\dots, m_{1,1}^{n-3}m_{2,2},m_{n,n}. \] 
\end{lemma}

\begin{proof}
Consider the family of equations
$E_{0,j}^{j+1}=0$, for $1\le j\le n-3$.
Computing directly, we find that
\begin{eqnarray*}
 E_{0,j} &=& g\mu_t(x_0,x_j)-\mu(gx_0,gx_j) \\
         &=& gx_{j+1}-\mu(m_{1,1}x_0+m_{2,1}x_1+\dots,m_{j+1,j+1}x_j+m_{j+2,j+1}x_{j+1}+\dots) \\
         &=& m_{j+2,j+2}x_{j+1}+m_{j+3,j+2}x_{j+2}+\dots-m_{1,1}m_{j+1,j+1}x_{j+1}+\dots
\end{eqnarray*}
Hence, $E_{0,j}^{j+1}=0$ is equivalent to $m_{j+2,j+2}=m_{1,1}m_{j+1,j+1}$.
Therefore, inductively we get that
\[ m_{3,3}=m_{1,1}m_{2,2},\dots,m_{j,j}=m_{1,1}^{j-2}m_{2,2},\dots,m_{n-1,n-1}=m_{1,1}^{n-3}m_{2,2},m_{n,n}. \]
\end{proof}

\begin{remark}\label{rmk:mnn}
 At this stage we said nothing about the last diagonal entry $m_{n,n}$.
 In case $n$ is odd, $E_{0,n-2}^{n-1}=-m_{1,1}m_{n-1,n-1}+m_{n,n}$, therefore it follows that $m_{n,n}=m_{1,1}^{n-2}m_{2,2}$.
 However, in case $n$ is even, $E_{0,n-2}^{n-1}=a_{4,9}m_{2,1}m_{n-1,n-1}-m_{1,1}m_{n-1,n-1}+m_{n,n}$.
\end{remark}

\begin{lemma}\label{lemma:m22}
If $\mu$ is in the open set of $\F^n$ given by the condition $a_{1,4}\ne 0$,
then $m_{2,2}=m_{1,1}^2$ and the diagonal entries of $[g]$ are
\[ m_{1,1},\ m_{1,1}^2,\ \dots,\ m_{1,1}^{n-1},m_{n,n}. \]
\end{lemma}
\begin{proof}
Let us consider $E_{1,2}=g\mu_t(x_1,x_2)-\mu(gx_1,gx_2)$.
On the one hand 
\begin{eqnarray*}
 \mu_t(x_1,x_2) &=& \mu_0(x_1,x_2)+\sum_\Delta a_{r,s}\psi_{r,s}(x_1,x_2)+t\varphi_D(x_1,x_2) \\
                &=& a_{1,4}x_4+a_{1,5}x_5+\dots+a_{1,n-1}x_{n-1}+tx_{m}
\end{eqnarray*}
where $m=n-3$ if $D=D^3$ or $m=n-4$ if $D=D^4$, and in any case $m\ge 5$.
And then 
\begin{eqnarray*}
 g\mu_t(x_1,x_2) &=& a_{1,4}m_{1,1}^3m_{2,2}x_4+a_{1,4}m_{6,5}x_5+a_{1,5}m_{1,1}^4m_{2,2}x_5+\dots \\
                 &=& a_{1,4}m_{1,1}^3m_{2,2}x_4+(a_{1,4}m_{6,5}+a_{1,5}m_{1,1}^4m_{2,2})x_5+\dots
\end{eqnarray*}
On the other hand
\begin{eqnarray*}
\mu(gx_1,gx_2) &=& \mu(m_{2,2}x_1+m_{3,2}x_2+\dots,m_{1,1}m_{2,2}x_2+m_{4,3}x_3+\dots) \\
               &=& a_{1,4}m_{1,1}m_{2,2}^2x_4+(a_{1,5}m_{1,1}m_{2,2}^2+a_{1,4}m_{2,2}m_{4,3})x_5+\dots
\end{eqnarray*}
Therefore, $E_{1,2}^4=0$ is  
\[ a_{1,4}m_{1,1}^3m_{2,2}=a_{1,4}m_{1,1}m_{2,2}^2. \]
If $a_{1,4}\ne 0$, this implies that $m_{2,2}=m_{1,1}^2$ and the result follows
from Lemma \ref{lemma:mii}.
\end{proof}

We end this section by showing that the diagonal entries of $[g]$ are all equal to 1
in a large open set of $\F^n$, in the cases we shall deal with.

\begin{lemma}\label{lemma:m11}
 In the Zariski open set $U$ of $\F^n$,
 \[ U=\{ a_{1,4}\ne 0, a_{1,5}\ne 0, 3a_{2,6}a_{1,5}(a_{1,4}-a_{2,6})-2a_{2,7}a_{1,4}^2\ne 0 \}, \]
 $m_{1,1}=1$ and $m_{2,1}=0$. 
\end{lemma}

\begin{proof}
Let us consider, in the given order, the following equations:
\[ E_{0,1}^3=0,\quad E_{0,2}^4=0,\quad E_{0,3}^5=0,\quad E_{0,4}^6=0, \]
and solve in each step $m_{4,3},m_{5,4},m_{6,5},m_{7,6}$.
We then get
\begin{eqnarray*}
 m_{4,3} &=& m_{1,1}m_{3,2} \\
 m_{5,4} &=& a_{1,4}m_{1,1}^3m_{2,1}+m_{1,1}^2m_{3,2} \\
 m_{6,5} &=& m_{1,1}(a_{1,4}m_{1,1}^3m_{2,1}+m_{1,1}^2m_{3,2})+a_{1,4}m_{2,1}m_{1,1}^4 \\
 m_{7,6} &=& m_{1,1}\big(m_{1,1}(a_{1,4}m_{1,1}^3m_{2,1}+m_{1,1}^2m_{3,2})+a_{1,4}m_{2,1}m_{1,1}^4\big) \\
         & & \quad +(-a_{2,6}+a_{1,4})m_{2,1}m_{1,1}^5
\end{eqnarray*}

Now consider
\begin{eqnarray*}
 E1 &=& a_{1,4}E_{0,4}^6-E_{1,3}^6 \\
    &=& -2a_{1,4}^2m_{1,1}^5m_{2,1}-a_{1,5}m_{1,1}^7+a_{1,5}m_{1,1}^6 \\
\end{eqnarray*}
From $E1=0$, it follows that
\begin{equation}\label{eqn:m21}
 m_{2,1}=\frac{m_{1,1}a_{1,5}(1-m_{1,1})}{2a_{1,4}^2}.
\end{equation}
Finally, consider
\begin{eqnarray*}
 E2 &=& (a_{1,4}-a_{2,6})E_{0,5}^7-E_{1,4}^7-m_{1,1}E1 \\
    &=& \frac{m_{1,1}^7(1-m_{1,1})\big(3a_{2,6}a_{1,5}(a_{1,4}-a_{2,6})-2a_{2,7}a_{1,4}^2\big)}{2a_{1,4}^2} 
\end{eqnarray*}
Hence from $E2=0$ and since $m_{1,1}\ne 0$, it follows that inside $U$, $m_{1,1}=1$. 
And from \eqref{eqn:m21} it follows that $m_{2,1}=0$.
\end{proof}

\begin{remark}[On the proof of Lemma \ref{lemma:m11}]
Notice that $E_{1,3}^6$ does not depend on $D$.
In fact $D^3(x_3)=x_{n-2}$ and $D^4(x_3)=x_{n-3}$, so with the choices made in $\eqref{eqn:mut}$, 
it follows that $n-2\ge 7$ and $n-3\ge 7$ respectively.
\end{remark}

\begin{lemma}\label{lemma:m32}
 In $U$, it holds in addition that
 \[ m_{n-1,n-2}=\dots=m_{4,3}=m_{3,2}. \]
\end{lemma}

\begin{proof}
From Lemmas \ref{lemma:mii}, \ref{lemma:m22} and \ref{lemma:m11}
together with Remark \ref{rmk:mnn} it follows that in $U$, $m_{2,1}=0$ and
all diagonal entries of $[g]$ are equal to 1.

We then have that
 \[ E_{0,1}^3=m_{4,3}-m_{3,2},\; E_{0,2}^4=m_{5,4}-m_{4,3},\;\dots,\; E_{0,n-4}^{n-2}=m_{n-1,n-2}-m_{n-2,n-3}, \]
 and therefore the result follows from the equations 
 \[ E_{0,1}^3=E_{0,2}^4=\dots=E_{0,n-4}^{n-2}=0. \]
\end{proof}

All together yields the following form for the matrix $[g]$,
that we shall use repeatedly afterwards. 

\begin{proposition}\label{prop:g}
 Let $n\ge 9$, $\mu\in\F^n$ and $\mu_t$ be the linear deformation of $\mu$ as in $\eqref{eqn:mut}$
 Let $g$ be an isomorphism between $\mu_t$ and $\mu$, and $[g]$
 its matrix with respect to an adapted basis $\{x_0,\dots,x_{n-1}\}$ for $\mu$ and $\mu_t$.
 
 Then in the (Zariski) open set of $\F^n$
 \[ U=\{ a_{1,4}\ne 0, a_{1,5}\ne 0, 3a_{2,6}a_{1,5}(a_{1,4}-a_{2,6})-2a_{2,7}a_{1,4}^2\ne 0 \}, \] 
 $[g]$ is of the form ($m_{3,2}=a$)
 \[ [g] = \begin{pmatrix}
  1 & 0 & 0 & \cdots & 0 & 0 \\
  0 & 1 &  0  & \cdots & 0 & 0 \\
 m_{3,1} & a & 1  & \cdots & 0 & 0 \\
 \vdots & \vdots & \ddots  & \ddots & \vdots & \vdots \\
 \vdots & \vdots & \vdots  & a & 1 & 0 \\
 m_{n,1} & m_{n,2} & m_{n,3} & \cdots & m_{n,n-1} & 1
 \end{pmatrix}.
 \]
\end{proposition}

\begin{remark}
The open set $U$ is big enough for our purpose.
$U$ intersects non-trivially all the irreducible components of $\F^n$, for $n=9,10,11$.
We shall show this in due time.
\end{remark}

\section{Non-trivial deformations}

In this section we prove that for $\mu$ in a dense open set of $\F^n$, for $n=9,10,11$, 
the linear deformations constructed in Section \ref{sec:deformations} are non-trivial.
It then follows that there are no rigid filiform Lie algebras of these dimensions.

\subsection{Dimension 9}

Let $\mu$ be a given filiform Lie bracket of dimension 9 and let 
$\{x_0, x_1, x_2, x_3, x_4, x_5, x_6, x_7, x_8 \}$ be a adapted basis for it.
Then, there are $a_{r,s}\in\C$, with $(r,s)\in\Delta_9$, such that: 
\begin{eqnarray*}
\mu &=& \mu_0 + a_{1,4} \psi_{1,4} + a_{1,5} \psi_{1,5} + a_{1,6} \psi_{1,6} + a_{1,7} \psi_{1,7} + a_{1,8} \psi_{1,8} \\
     &&  \phantom{\mu_0 + a_{1,4} \psi_{1,4} + a_{1,5} \psi_{1,5} }  
          + a_{2,6} \psi_{2,6} + a_{2,7} \psi_{2,7} + a_{2,8} \psi_{2,8} \\
     &&  \phantom{\mu_0 + a_{1,4} \psi_{1,4} + a_{1,5} \psi_{1,5} + a_{1,6} \psi_{1,6} + a_{1,7} \psi_{1,7}}
          + a_{3,8} \psi_{3,8}.
\end{eqnarray*}
So that $\mu(x_0, x_j) = x_{j+1}$, for $1\le j\le 7$, and
\begin{eqnarray*}
\mu(x_1, x_2) &=& a_{1,4} x_4 + a_{1,5} x_5 + a_{1,6} x_6 + a_{1,7} x_7 + a_{1,8} x_8 \\
\mu(x_1, x_3) &=& a_{1,4} x_5 + a_{1,5} x_6 + a_{1,6} x_7 + a_{1,7} x_8\\
\mu(x_1, x_4) &=& ( a_{1,4} - a_{2,6}) x_6 + ( a_{1,5} - a_{2,7}) x_7 + ( a_{1,6} - a_{2,8}) x_8\\
\mu(x_1, x_5) &=& ( a_{1,4} - 2 a_{2,6}) x_7 + ( a_{1,5} - 2 a_{2,7}) x_8 \\
\mu(x_1, x_6) &=& ( a_{1,4} - 3 a_{2,6} + a_{3,8} ) x_8\\
\mu(x_2, x_3) &=& a_{2,6} x_6 + a_{2,7} x_7 + a_{2,8} x_8\\
\mu(x_2, x_4) &=& a_{2,6} x_7 + a_{2,7} x_8\\
\mu(x_2, x_5) &=& ( a_{2,6} - a_{3,8} ) x_8\\
\mu(x_3, x_4) &=& a_{3,8} x_8.
\end{eqnarray*}

Let $\mu_t$ be the linear deformation of $\mu$ defined in Section \ref{sec:deformations} 
associated to $D^3$, that is (see Remark \ref{rmk:phiD})
\[ \mu_t=\mu + t\psi_{1,6}. \]

\begin{proposition}\label{prop:dim9}
 Let $\mu\in\F^9$ be a given filiform Lie bracket of dimension 9
 and let $U^9$ be the following Zariski open set of $\F^9$,
 \begin{eqnarray*}
 U^9 &=& U \cap U' \\
   &=& \{ a_{1,4}\ne 0, a_{1,5}\ne 0, 3a_{2,6}a_{1,5}(a_{1,4}-a_{2,6})\ne 2a_{2,7}a_{1,4}^2 \} \\
    && \quad \cap \{2a_{2,6}-a_{1,4}\ne 0, a_{3,8}\ne 0  \}.
 \end{eqnarray*}
 If $\mu\in U^9$, then the filiform Lie bracket $\mu_t=\mu+t\psi_{1,6}$ is not isomorphic to $\mu$, except for $t=0$. 
\end{proposition}

\begin{proof}
 We proceed as we did it in Section \ref{subsec:isos}.
 So let $g$ be an isomorphism from $\mu_t$ to $\mu$ and let $[g]$
 be its matrix with respect to a standard basis $\{x_0,x_1,\dots,x_8\}$.
 Then $[g]$ is as in Proposition \ref{prop:g}. 
 
 From equations
 \[ E_{0,1}^4=0,\ E_{0,2}^5=0,\ E_{0,3}^6=0,\ E_{0,4}^7=0, \]
 we get that
 \begin{eqnarray*}
  m_{5,3} &=& -a_{1,4}m_{3,1}+m_{4,2} \\
  m_{6,4} &=& -a_{1,4}m_{3,1}+m_{4,2} \\
  m_{7,5} &=& -a_{1,4}m_{3,1}+ a_{2,6}m_{3,1}+m_{4,2} \\
  m_{8,6} &=& -a_{1,4}m_{3,1}+ 2a_{2,6}m_{3,1}+m_{4,2} \\
 \end{eqnarray*}
Let
\begin{eqnarray*}
  E &=& a_{2,7}E_{1,5}^8-(a_{1,4}-2a_{2,6})E_{2,3}^8+(a_{1,4}a_{2,6}-2a_{2,6}^2)E_{0,5}^8 \\
     &=& (-2a_{2,6}+a_{1,4})\big((2a_{1,4}a_{3,8}-3a_{2,6}^2+a_{2,6}a_{3,8})m_{3,1}+a_{3,8}(m_{3,2}^2-2m_{4,2})  \big)
\end{eqnarray*} 
and notice that $2a_{1,4}a_{3,8}-3a_{2,6}^2+a_{2,6}a_{3,8}=0$ is the defining equation of $\F^9$.
 Hence
 \[ E=(-2a_{2,6}+a_{1,4})a_{3,8}(m_{3,2}^2-2m_{4,2}).\]
 Therefore, in $U^9$, it is 
 \[ m_{4,2}=\frac{1}{2}m_{3,2}^2. \]
 Finally, since
 \[ E_{1,2}^6=-a_{2,6}m_{3,2}^2+2a_{2,6}m_{4,2}+t, \]
 it follows that, in $U^9$, $t=0$.
\end{proof}

\subsection{Dimension 10}

Let $\mu$ be a given filiform Lie bracket of dimension 10 and let 
$\{x_0, x_1, x_2, x_3, x_4, x_5, x_6, x_7, x_8, x_9 \}$ be a adapted basis for it.
Then, there are $a_{r,s}\in\C$, with $(r,s)\in\Delta_{10}$, such that: 
\begin{eqnarray*}
\mu &=& \mu_0 + a_{1,4} \psi_{1,4} + a_{1,5} \psi_{1,5} + a_{1,6} \psi_{1,6}+a_{1,7}\psi_{1,7}+a_{1,8}\psi_{1,8}+a_{1,9}\psi_{1,9} \\
     &&  \phantom{\mu_0 + a_{1,4} \psi_{1,4} + a_{1,5} \psi_{1,5} }  
          + a_{2,6} \psi_{2,6} + a_{2,7} \psi_{2,7} + a_{2,8} \psi_{2,8} + a_{2,9}\psi_{2,9} \\
     &&  \phantom{\mu_0 + a_{1,4} \psi_{1,4} + a_{1,5} \psi_{1,5} + a_{1,6} \psi_{1,6} + a_{1,7} \psi_{1,7}}
          + a_{3,8} \psi_{3,8} + a_{3,9}\psi_{3,9} \\
     &&  \phantom{\mu_0 + a_{1,4} \psi_{1,4} + a_{1,5} \psi_{1,5} + a_{1,6} \psi_{1,6}+a_{1,7}\psi_{1,7}+a_{1,8}\psi_{1,8}}
          + a_{4,9}\psi_{4,9}
\end{eqnarray*}
So that $\mu(x_0, x_j) = x_{j+1}$, for $1\le j\le 8$, and
\begin{eqnarray*}
\mu(x_1, x_2) &=& a_{1,4} x_4 + a_{1,5} x_5 + a_{1,6} x_6 + a_{1,7} x_7 + a_{1,8} x_8 + a_{1,9} x_9 \\
\mu(x_1, x_3) &=& a_{1,4} x_5 + a_{1,5} x_6 + a_{1,6} x_7 + a_{1,7} x_8 + a_{1,8} x_9\\
\mu(x_1, x_4) &=& ( a_{1,4} - a_{2,6}) x_6 + ( a_{1,5} - a_{2,7}) x_7 + ( a_{1,6} - a_{2,8}) x_8 + (a_{1,7}-a_{2,9}) x_9 \\
\mu(x_1, x_5) &=& ( a_{1,4} - 2 a_{2,6}) x_7 + ( a_{1,5} - 2 a_{2,7}) x_8 + (a_{1,6}-2a_{2,8}) x_9 \\
\mu(x_1, x_6) &=& ( a_{1,4} - 3 a_{2,6} + a_{3,8} ) x_8 + (a_{1,5} -3a_{2,7}+a_{3,9}) x_9 \\
\mu(x_1, x_7) &=& ( -4 a_{2,6} + 3 a_{3,8} + 2 a_{4,9} ) x_9 \\
\mu(x_1, x_8) &=& -a_{4,9} x_9 \\
\mu(x_2, x_3) &=& a_{2,6} x_6 + a_{2,7} x_7 + a_{2,8} x_8 + a_{2,9} x_9 \\
\mu(x_2, x_4) &=& a_{2,6} x_7 + a_{2,7} x_8 + a_{2,8} x_9 \\
\mu(x_2, x_5) &=& ( a_{2,6} - a_{3,8} ) x_8 + (a_{2,7}-a_{3,9}) x_9 \\
\mu(x_2, x_6) &=& ( 3 a_{2,6} - 2 a_{3,8} ) x_9 \\ 
\mu(x_2, x_7) &=& a_{4,9} x_9 \\
\mu(x_3, x_4) &=& a_{3,8} x_8 + a_{3,9} x_9 \\
\mu(x_3, x_5) &=& a_{3,8} x_9 \\
\mu(x_3, x_6) &=& -a_{4,9} x_9 \\
\mu(x_4, x_5) &=& a_{4,9} x_9.
\end{eqnarray*}

Let $\mu_t$ be the linear deformation of $\mu$ defined in Section \ref{sec:deformations}
associated to $D^3$, that is (see Remark \ref{rmk:phiD})
\[ \mu_t=\mu + t\psi_{1,7}. \]

To prove that $\mu_t$ is a non trivial deformation of $\mu$ 
we proceed as we did it in Section \ref{subsec:isos}.
So let $g$ be an isomorphism from $\mu_t$ to $\mu$ and let $[g]$
be its matrix with respect to a standard basis $\{x_0,x_1,\dots,x_9\}$.
Then $[g]$ is as in Proposition \ref{prop:g}.

\begin{proposition}\label{prop:dim10}
 Let $\mu\in \F^{10}$ be a given filiform Lie bracket of dimension 10
 and let $U^{10}$ be the following Zariski open set of $\F^{10}$,
 \begin{eqnarray*}
 U^{10} &=& U \cap U' \\
   &=& \{ a_{1,4}\ne 0, a_{1,5}\ne 0, 3a_{2,6}a_{1,5}(a_{1,4}-a_{2,6})\ne 2a_{2,7}a_{1,4}^2 \} \\
    && \quad \cap \{a_{2,6}\ne 0, a_{3,8}\ne 0, a_{1,4}^2+a_{2,7}a_{4,9}\ne 0, 15a_{1,4}^2-a_{2,7}a_{4,9}\ne 0  \}.
 \end{eqnarray*}
 If $\mu\in U^{10}$, then the filiform Lie bracket $\mu_t=\mu+t\psi_{1,7}$ is not isomorphic to $\mu$, except for $t=0$. 
\end{proposition}

\begin{proof}
 We show first that 
 \[ m_{4,2}=\frac{1}{2}m_{3,2}^2. \]
 For this, let us consider 
 \begin{eqnarray*}
  E_{0,1}^4 &=& a_{1,4}m_{3,1}-m_{4,2}+m_{5,3} \\
  E_{0,2}^5 &=& a_{1,4}m_{3,1}-m_{4,2}+m_{6,4} \\
  E_{0,3}^6 &=& a_{1,4}m_{3,1}-a_{2,6}m_{3,1}-m_{4,2}+m_{7,5};
 \end{eqnarray*}
from the corresponding equations $E_{*,*}^*=0$ we get that
 \begin{eqnarray*}
  m_{5,3} &=& -a_{1,4}m_{3,1}+m_{4,2} \\
  m_{6,4} &=& -a_{1,4}m_{3,1}+m_{4,2} \\
  m_{7,5} &=& -a_{1,4}m_{3,1}+a_{2,6}m_{3,1}+m_{4,2}
 \end{eqnarray*}
Now, $E_{1,2}^6=-a_{2,6}m_{3,2}^2+2a_{2,6}m_{4,2}$ and hence from 
the equation $E_{1,2}^6=0$ and since $a_{2,6}\ne 0$, it follows what we claimed.

From now on we find it convenient to work in each of the irreducible components $C^{10}_1$, $C^{10}_2$ and $C^{10}_3$ of $\F^{10}$ separately
(see \ref{subsec:F}).

\

\noindent {\sc (i)} Assume $\mu\in C^{10}_1$. In particular $a_{4,9}=0$.

We have that $E_{0,7}^9=m_{10,9}-m_{3,2}$, and therefore it is (also) $m_{10,9}=m_{3,2}$.
And by considering the equations
\[ E_{0,1}^4=0,\quad E_{0,2}^5=0,\quad E_{0,3}^6=0,\quad E_{0,4}^7=0,\text{ and } E_{0,5}^8=0 \]
we get
\begin{eqnarray*}
 m_{5,3} &=& \frac{1}{2}m_{3,2}^2-a_{1,4}m_{3,1} \\
 m_{6,4} &=& \frac{1}{2}m_{3,2}^2-a_{1,4}m_{3,1} \\
 m_{7,5} &=& \frac{1}{2}m_{3,2}^2+(a_{2,6}-a_{1,4})m_{3,1} \\
 m_{8,6} &=& \frac{1}{2}m_{3,2}^2+(2a_{2,6}-a_{1,4})m_{3,1} \\
 m_{9,7} &=& \frac{1}{2}m_{3,2}^2+(3a_{2,6}-a_{1,4}-a_{3,8})m_{3,1}
\end{eqnarray*}
Then we have that  
\[ E_{2,3}^8+(-3a_{2,6}^2+a_{2,6}a_{3,8}+a_{1,4}a_{3,8})m_{3,1}=-a_{3,8}a_{1,4}m_{3,1}, \]
and since $a_{3,8}a_{1,4}\ne 0$ in $U^{10}$, it follows that
\[ m_{3,1}=0. \]

Now, from the equations
\[ E_{0,1}^5=0,\quad E_{0,2}^6=0,\quad E_{0,3}^7=0 \text{ and } E_{0,4}^8=0 \]
we get that
\begin{eqnarray*}
 m_{6,3} &=& m_{5,2}-a_{1,4}m_{4,1} \\
 m_{7,4} &=& m_{5,2}-(a_{1,4}+a_{2,6})m_{4,1} \\
 m_{8,5} &=& m_{5,2}-(a_{1,4}+a_{2,6})m_{4,1} \\
 m_{9,6} &=& m_{5,2}+(a_{3,8}-a_{1,4}-a_{2,6})m_{4,1}
\end{eqnarray*}
We then have that
\[ 0=E_{1,2}^7-E_{1,3}^8+(-3a_{2,6}^2+a_{2,6}a_{3,8}+a_{1,4}a_{3,8})m_{4,1}=-a_{1,4}a_{3,8}m_{4,1}, \]
and since $a_{3,8}a_{1,4}\ne 0$ in $U^{10}$, it follows that
\[ m_{4,1}=0. \]
Given this, we consider the equations
\[ E_{0,5}^9=0 \text{ \ and \ } E_{0,6}^9=0  \]
from which we get that
\[ m_{10,7}=m_{5,2} \text{ \ and \ } m_{10,8}=\frac{1}{2}m_{3,2}^2. \]
Finally,
\[ E_{1,2}^7=t+3a_{2,6}m_{5,2}-\frac{1}{2}a_{2,6}m_{3,2}^2, \]
and
\[ E_{1,4}^9=t+3a_{2,6}m_{5,2}-\frac{1}{2}a_{2,6}m_{3,2}^2-3a_{3,8}m_{5,2}+\frac{1}{2}a_{3,8}m_{3,2}^2. \]
Hence
\[ E_{1,2}^7-E_{1,4}^9=3a_{3,8}m_{5,2}-\frac{1}{2}a_{3,8}m_{3,2}^2, \]
and since $a_{3,8}\ne 0$ in $U^{10}$, it follows that
\[ m_{5,2}=\frac{1}{6}m_{3,2}^2 \]
and therefore
\[ t=E_{1,2}^7=0. \]

\noindent {\sc (ii)} Assume $\mu\in C_2$. In particular $a_{2,6}=a_{3,8}=a_{1,4}$.

From the equations
\[ E_{0,1}^4=0,\quad E_{0,2}^5=0,\quad E_{0,3}^6=0,\quad E_{0,4}^7=0,\text{ \ and \ } E_{0,5}^8=0 \]
we get
\begin{eqnarray*}
 m_{5,3} &=& \frac{1}{2}m_{3,2}^2-a_{1,4}m_{3,1} \\
 m_{6,4} &=& \frac{1}{2}m_{3,2}^2-a_{1,4}m_{3,1} \\
 m_{7,5} &=& \frac{1}{2}m_{3,2}^2 \\
 m_{8,6} &=& \frac{1}{2}m_{3,2}^2+a_{1,4}m_{3,1} \\
 m_{9,7} &=& \frac{1}{2}m_{3,2}^2+a_{1,4}m_{3,1}
\end{eqnarray*}
Then we have that  
\[ 0=E_{2,3}^8+(-3a_{2,6}^2+a_{2,6}a_{3,8}+a_{1,4}a_{3,8})m_{3,1}=-a_{1,4}^2m_{3,1}, \]
and since $a_{1,4}\ne 0$ in $U^{10}$, it follows that
\[ m_{3,1}=0. \]

By considering the equations
\[ E_{0,1}^5=0,\quad E_{0,2}^6=0,\quad E_{0,3}^7=0 \text{ \ and \ } E_{0,4}^8=0 \]
we get that
\begin{eqnarray*}
 m_{6,3} &=& m_{5,2}-a_{1,4}m_{4,1} \\
 m_{7,4} &=& m_{5,2}-2a_{1,4}m_{4,1} \\
 m_{8,5} &=& m_{5,2}-2a_{1,4}m_{4,1} \\
 m_{9,6} &=& m_{5,2}-a_{1,4}m_{4,1}
\end{eqnarray*}
Now we have that
\[ 0=E_{1,2}^7-E_{1,3}^8+(-3a_{2,6}^2+a_{2,6}a_{3,8}+a_{1,4}a_{3,8})m_{4,1}=-a_{1,4}^2m_{4,1}, \]
and since $a_{1,4}\ne 0$ in $U^{10}$, it follows that
\[ m_{4,1}=0. \]

Given this, we continue considering the equations
\[ E_{0,5}^9=0, \quad E_{0,6}^9=0 \text{ and } E_{0,7}^9=0 \]
from which we get that
\begin{eqnarray*}
  m_{10,7} &=& m_{5,2} + a_{4,9}m_{5,1} \\
  m_{10,8} &=& \frac{1}{2}m_{3,2}^2 \\
  m_{10,9} &=& m_{3,2}.
\end{eqnarray*}
And considering the equations
\[ E_{0,1}^6=0, \quad E_{0,2}^7=0 \text{ and } E_{0,3}^8=0 \]
we get 
\begin{eqnarray*}
 m_{7,3} &=& m_{6,2} \\
 m_{8,4} &=& m_{6,2} -a_{1,4}m_{5,1} \\
 m_{9,5} &=& m_{6,2}-2a_{1,4}m_{5,1}.
\end{eqnarray*}
Finally,
\[ E_{1,4}^9+E_{2,3}^9=t-\frac{1}{2}a_{1,4}(m_{3,2}^2-6m_{5,2}), \]
and
\[ 0 = a_{1,4}E_{2,3}^9+a_{4,9}E_{1,2}^8 = -\frac{1}{2}(m_{3,2}^2-6m_{5,2})(a_{1,4}^2+a_{2,7}a_{4,9}). \]
Hence, since $a_{1,4}^2+a_{2,7}a_{4,9}\ne 0$ in $U^{10}$, it follows that $m_{3,2}^2-6m_{5,2}=0$ and $t=0$.

(3) Assume $\mu\in C_3$.

We have that $a_{2,6}=-a_{1,4}$ and $a_{3,8}=3a_{1,4}$.

By considering the equations
\[ E_{0,1}^4=0,\quad E_{0,2}^5=0,\quad  E_{0,3}^6=0,\quad E_{0,4}^7=0,\text{ and } E_{0,5}^8=0 \]
we get
\begin{eqnarray*}
 m_{5,3} &=& \frac{1}{2}m_{3,2}^2-a_{1,4}m_{3,1} \\
 m_{6,4} &=& \frac{1}{2}m_{3,2}^2-a_{1,4}m_{3,1} \\
 m_{7,5} &=& \frac{1}{2}m_{3,2}^2-2a_{1,4}m_{3,1} \\
 m_{8,6} &=& \frac{1}{2}m_{3,2}^2-3a_{1,4}m_{3,1} \\
 m_{9,7} &=& \frac{1}{2}m_{3,2}^2-7a_{1,4}m_{3,1}
\end{eqnarray*}
Then we have that  
\[ 0=E_{2,3}^8+(-3a_{2,6}^2+a_{2,6}a_{3,8}+a_{1,4}a_{3,8})m_{3,1}=-3a_{1,4}^2m_{3,1}, \]
and since $a_{1,4}\ne 0$ in $U^{10}$, it follows that
\[ m_{3,1}=0. \]

By considering the equations
\[ E_{0,1}^5=0,\quad E_{0,2}^6=0,\quad E_{0,3}^7=0 \text{ and } E_{0,4}^8=0 \]
we get that
\begin{eqnarray*}
 m_{6,3} &=& m_{5,2}-a_{1,4}m_{4,1} \\
 m_{7,4} &=& m_{5,2} \\
 m_{8,5} &=& m_{5,2} \\
 m_{9,6} &=& m_{5,2}+3a_{1,4}m_{4,1}
\end{eqnarray*}
Now we have that
\[ 0=E_{1,2}^7-E_{1,3}^8+(-3a_{2,6}^2+a_{2,6}a_{3,8}+a_{1,4}a_{3,8})m_{4,1}=-3a_{1,4}^2m_{4,1}, \]
and since $a_{1,4}\ne 0$ in $U^{10}$, it follows that
\[ m_{4,1}=0. \]

Given this, we continue considering the equations
\[ E_{0,5}^9=0, \quad E_{0,6}^9=0 \text{ and } E_{0,7}^9=0 \]
from which we get that
\begin{eqnarray*}
  m_{10,7} &=& m_{5,2} + a_{4,9}m_{5,1} \\
  m_{10,8} &=& \frac{1}{2}m_{3,2}^2 \\
  m_{10,9} &=& m_{3,2}.
\end{eqnarray*}
And considering the equations
\[ E_{0,1}^6=0, \quad E_{0,2}^7=0 \text{ and } E_{0,3}^8=0 \]
we get 
\begin{eqnarray*}
 m_{7,3} &=& m_{6,2}-2a_{1,4}m_{5,1} \\
 m_{8,4} &=& m_{6,2} -a_{1,4}m_{5,1} \\
 m_{9,5} &=& m_{6,2}-4a_{1,4}m_{5,1}.
\end{eqnarray*}
Finally,
\[ E_{2,1}^7=t-\frac{1}{2}a_{1,4}(m_{3,2}^2-6a_{1,4}m_{5,2}), \]
and
\[ 0 = 5a_{1,4}E_{2,3}^9-a_{4,9}E_{1,2}^8 = -\frac{1}{2}(m_{3,2}^2-6m_{5,2})(a_{1,4}^2+a_{2,7}a_{4,9}). \]
Hence, since $15a_{1,4}^2-a_{2,7}a_{4,9}\ne 0$ in $U^{10}$, it follows that $m_{3,2}^2-6m_{5,2}=0$ and $t=0$.

\end{proof}

\subsection{Dimension 11}

Let $\mu$ be a given filiform Lie bracket of dimension 11 and let 
$\{x_0, x_1, x_2, x_3, x_4, x_5, x_6, x_7, x_8, x_9, x_{10} \}$ be a adapted basis for it.
Then, there are $a_{r,s}\in\C$, with $(r,s)\in\Delta_{11}$, such that: 
\begin{eqnarray*}
\mu &=& \mu_0 + a_{1,4} \psi_{1,4} + a_{1,5} \psi_{1,5} + a_{1,6} \psi_{1,6}+a_{1,7}\psi_{1,7}+a_{1,8}\psi_{1,8}+a_{1,9}\psi_{1,9}+a_{1,10}\psi_{1,10} \\
     &&  \phantom{\mu_0 + a_{1,4} \psi_{1,4} + a_{1,5} \psi_{1,5} }  
          + a_{2,6} \psi_{2,6} + a_{2,7} \psi_{2,7} + a_{2,8} \psi_{2,8} + a_{2,9}\psi_{2,9} +a_{2,10}\psi_{2,10} \\
     &&  \phantom{\mu_0 + a_{1,4} \psi_{1,4} + a_{1,5} \psi_{1,5} + a_{1,6} \psi_{1,6} + a_{1,7} \psi_{1,7}}
          + a_{3,8} \psi_{3,8} + a_{3,9}\psi_{3,9} +a_{3,10}\psi_{3,10} \\
     &&  \phantom{\mu_0 + a_{1,4} \psi_{1,4} + a_{1,5} \psi_{1,5} + a_{1,6} \psi_{1,6}+a_{1,7}\psi_{1,7}+a_{1,8}\psi_{1,8}+a_{1,9}\psi_{1,9}}
          +a_{4,10}\psi_{4,10}
\end{eqnarray*}
So that $\mu(x_0, x_j) = x_{j+1}$, for $1\le j\le 10$, and

\begin{eqnarray*}
\mu(x_1, x_2) &=& a_{1,4} x_4 + a_{1,5} x_5 + a_{1,6} x_6 + a_{1,7} x_7 + a_{1,8} x_8 + a_{1,9} x_9 + a_{1,10} x_{10}\\
\mu(x_1, x_3) &=& a_{1,4} x_5 + a_{1,5} x_6 + a_{1,6} x_7 + a_{1,7} x_8 + a_{1,8} x_9 + a_{1,9} x_{10}\\
\mu(x_1, x_4) &=& ( a_{1,4} - a_{2,6}) x_6 + ( a{1,5} - a_{2,7}) x_7 + ( a_{1,6} - a_{2,8}) x_8 + (a_{1,7} - a_{2,9}) x_9 \\
&+& ( a_{1,8} - a_{2,10}) x_{10}\\
\mu(x_1, x_5) &=& (a_{1,4} - 2 a_{2,6}) x_7 + ( a_{1,5} - 2 a_{2,7}) x_8 + ( a_{1,6} - 2 a_{2,8}) x_9 + ( a_{1,7} - 2 a_{2,9}) x_{10}\\
\mu(x_1, x_6) &=& ( a_{1,4} - 3 a_{2,6} + a_{3,8} ) x_8 + ( a_{1,5} - 3 a_{2,7} + a_{3,9} ) x_9 + ( a_{1,6} - 3 a_{2,8} + a_{3,10} ) x_{10} \\
\mu(x_1, x_7) &=& ( a_{1,4} - 4 a_{2,6} + 3 a_{3,8} ) x_9 + ( a_{1,5} - 4 a_{2,7} + 3 a_{3,9} ) x_{10} \\
\mu(x_1, x_8) &=& ( a_{1,4} - 5 a_{2,6} + 6 a_{3,8} - a_{4,10} ) x_{10} \\
\mu(x_2, x_3) &=& a_{2,6} x_6 + a_{2,7} x_7 + a_{2,8} x_8 + a_{2,9} x_9 + a_{2,10} x_{10} \\
\mu(x_2, x_4) &=& a_{2,6} x_7 + a_{2,7} x_8 + a_{2,8} x_9 + a_{2,9} x_{10}\\
\mu(x_2, x_5) &=& ( a_{2,6} - a_{3,8} ) x_8 + ( a_{2,7} - a_{3,9} ) x_9 + ( a_{2,8} - a_{3,10} ) x_{10}\\
\mu(x_2, x_6) &=& ( a_{2,6} - 2 a_{3,8} ) x_9 + ( a_{2,7} - 2 a_{3,9} ) x_{10}\\
\mu(x_2, x_7) &=& ( a_{2,6} - 2 a_{3,8} + a_{4,10}) x_{10}\\
\mu(x_3, x_4) &=& a_{3,8} x_8 + a_{3,9} x_9 + a_{3,10} x_{10}\\
\mu(x_3, x_5) &=& a_{3,8} x_9 + a_{3,9} x_{10},\\
\mu(x_3, x_6) &=& (a_{3,8} - a_{4,10}) x_{10}\\
\mu(x_4, x_5) &=& a_{4,10} x_{10}
\end{eqnarray*}

Let $\mu_t$ be the linear deformation of $\mu$ defined in Section \ref{sec:deformations}
associated to $D^4$, that is (see Remark \ref{rmk:phiD})
\[ \mu_t=\mu + t\psi_{1,7}. \]

To prove that $\mu_t$ is a non trivial deformation of $\mu$ 
we proceed as we did it in Section \ref{subsec:isos}.
So let $g$ be an isomorphism from $\mu_t$ to $\mu$ and let $[g]$
be its matrix with respect to a standard basis $\{x_0,x_1,\dots,x_{10}\}$.
Then $[g]$ is as in Proposition \ref{prop:g}.

\begin{proposition}\label{prop:dim11}
 Let $\mu\in \F^{11}$ be a given filiform Lie bracket of dimension 11
 and let $U^{11}$ be the following Zariski open set of $\F^{11}$,
 \begin{eqnarray*}
 U^{11} &=& U \cap U' \\
   &=& \{ a_{1,4}\ne 0, a_{1,5}\ne 0, 3a_{2,6}a_{1,5}(a_{1,4}-a_{2,6})\ne 2a_{2,7}a_{1,4}^2 \} \\
    && \quad \cap \{a_{2,6}\ne 0, a_{3,8}\ne 0  \}.
 \end{eqnarray*}
 If $\mu\in U^{11}$, then the filiform Lie bracket $\mu_t=\mu+t\psi_{1,7}$ is not isomorphic to $\mu$, except for $t=0$. 
\end{proposition}

\begin{proof}
By considering the equations
\[ E_{0,1}^4=0,\quad E_{0,2}^5=0,\quad E_{0,3}^6=0,\quad E_{0,4}^7=0, \]
\[ E_{0,5}^8=0,\quad E_{0,6}^9=0,\text{ and } E_{0,7}^{10}=0 \]
we get
\begin{eqnarray*}
 m_{5,3} &=& -a_{1,4}m_{3,1}+m_{4,2} \\
 m_{6,4} &=& -a_{1,4}m_{3,1}+m_{4,2} \\
 m_{7,5} &=& (-a_{1,4}+a_{2,6})m_{3,1}+m_{4,2} \\
 m_{8,6} &=& (-a_{1,4}+2a_{2,6})m_{3,1}+m_{4,2} \\
 m_{9,7} &=& (-a_{1,4}+3a_{2,6}-a_{3,8})m_{3,1}+m_{4,2} \\
 m_{10,8} &=& (-a_{1,4}+4a_{2,6}-3a_{3,8})m_{3,1}+m_{4,2} \\
 m_{11,9} &=& (-a_{1,4}+5a_{2,6}-6a_{3,8}+a_{4,10})m_{3,1}+m_{4,2}
 \end{eqnarray*}
 
And by considering the equations
\[ E_{0,1}^5=0,\quad E_{0,2}^6=0,\quad E_{0,3}^7=0,\quad E_{0,4}^8=0,\text{ and } E_{0,5}^{9}=0 \]
we get
\begin{eqnarray*}
 m_{6,3} &=& -a_{1,5}m_{3,1}-a_{1,4}m_{4,1}+m_{5,2} \\
 m_{7,4} &=& -a_{1,5}m_{3,1}-(a_{1,4}+a_{2,6})m_{4,1}+m_{5,2}+a_{2,6}m_{3,1}m_{3,2} \\
 m_{8,5} &=& (-a_{1,5}+a_{2,7})m_{3,1}-(a_{1,4}+a_{2,6})m_{4,1}+m_{5,2}+2a_{2,6}m_{3,1}m_{3,2} \\
 m_{9,6} &=& (-a_{1,5}+2a_{2,7})m_{3,1}-(a_{1,4}+a_{2,6}-a_{3,8})m_{4,1}+m_{5,2}+(3a_{2,6}-a_{3,8})m_{3,1}m_{3,2} \\
 m_{10,7} &=& (-a_{1,5}+3a_{2,7}-a_{3,9})m_{3,1}-(a_{1,4}+a_{2,6}-2a_{3,8})m_{4,1}+m_{5,2}+(4a_{2,6}-3a_{3,8})m_{3,1}m_{3,2} 
 \end{eqnarray*}
  
Then we have that  
\[ E_{1,2}^6=-a_{2,6}(m_{3,2}^2-2m_{4,2}), \]
and since $a_{2,6}\ne 0$ in $U^{11}$, it follows that
\[ m_{4,2}=\frac{1}{2}m_{3,2}^2. \]

Now
\begin{eqnarray*}
 E_{1,2}^7 &=& -3a_{1,4}a_{2,6}m_{4,1}+3a_{2,6}m_{5,2}+3a_{1,4}a_{2,6}m_{3,1}m_{3,2}-\frac{1}{2}a_{2,6}m_{3,2}^2+t \\
           &=& t-3a_{2,6}\big(a_{1,4}m_{4,1}-m_{5,2}-a_{1,4}m_{3,1}m_{3,2}+\frac{1}{6}m_{3,2}^2 \big)
\end{eqnarray*}
and 
\[ E_{2,3}^9+P_1m_{3,1}m_{3,2}+P_3m_{3,1}=
  -3a_{3,8}\big(a_{1,4}m_{4,1}-m_{5,2}-a_{1,4}m_{3,1}m_{3,2}+\frac{1}{6}m_{3,2}^2 \big)
\]
where 
\begin{eqnarray*}
P_1 &=& 2a_{1,4}a_{3,8}-3a_{2,6}^2+a_{2,6}a_{3,8}  \\
P_3 &=& 2a_{1,4}a_{3,9}+3a_{1,5}a_{3,8}-7a_{2,6}a_{2,7}+a_{2,6}a_{3,9}+3a_{2,7}a_{3,8}  
\end{eqnarray*}
are defining equations of $\F^{11}$ (see ).

Therefore, since $a_{3,8}\ne 0$ in $U^{11}$, $E_{1,2}^7=t$ and $t=0$.
\end{proof}

\section{The main results}

In the previous section we proved that all filiform Lie algebras in certain open sets of $\F^9$,
$\F^{10}$ and $\F^{11}$, have a non-trivial deformation.
These open sets are big enough to prove that there are no rigid filiform Lie algebras inside these varieties.

We identify a point in $\F^{n}$ with the tuple of parameters in $\Delta_n$ ordered lexicographically.
Recall that $\Delta_9$ has 9 parameters, $\Delta_{10}$ has 13 and $\Delta_{11}$ has 16.

\begin{theorem}
 There are no rigid filiform Lie algebras of dimension 9, 10 and 11.
 Moreover, in any Euclidean neighborhood of a filiform Lie bracket $\mu\in\F^n$,
 with $n=9,10,11$, there is another non-isomorphic filiform Lie bracket $\nu$.
\end{theorem}

\begin{proof}
Let $U^9\subseteq\F^{9}$, $U^{10}\subseteq\F^{10}$ and $U^{11}\subseteq\F^{11}$ be the open sets
given, respectively, in Propositions \ref{prop:dim9}, \ref{prop:dim10} and \ref{prop:dim11}.

Each one of these is Zariski dense in the corresponding variety, since they intersect non-trivially
each of the irreducible components of the variety, and therefore they are Euclidean dense.
This follows from:
\begin{enumerate}
 \item $\F^{9}$ is irreducible and
 \[ ( 1, -1, 0, 0, 0, 1, 1, 0, 1)\in U^{9}\cap\F^{9}. \]
   
 \item $\F^{10}$ has three irreducible components $C^{10}_1$, $C^{10}_2$ and $C^{10}_3$ (see \S\ref{subsec:F}), and
 \[ ( 1, -16, 0, 0, 0, 0, 4, 0, 0, 0, 8, 64,0) \in U^{10}\cap C^{10}_1 \]
 \[ ( 1, 4, 0, 0, 0, 0, 1, 3, 0, 0, 1, 0,0) \in U^{10}\cap C^{10}_2 \]
 \[ ( 1, -16, 0, 0, 0, 0, -1, 9, 0, 0, 3, 0, 0) \in U^{10}\cap C^{10}_3. \]
 
 \item $\F^{11}$ has two irreducible components $C^{11}_1$ and $C^{11}_2$ (see \S\ref{subsec:F}), and
 \[ (1,\frac{1}{6}, 0, 0, 0, 0, 0, 4, 1, - \frac{5}{12}, 0, 0, 8, 0, 0, \frac{128}{5}) 
      \in U^{11}\cap C^{11}_1 \]
 \[ (1, 1, 0, 0, 0, 0, 0, 4, 0, 0, 0, 0, 8, -6, 3, \frac{128}{5}) \in U^{11}\cap C^{11}_2 \]
\end{enumerate}
Let $n=9,10,11$.
Given $\nu\in\F^n$ and a neighborhood $V$ of $\nu$, take $\mu\in V\cap U^n\ne\emptyset$.
If $\mu\not\simeq\nu$, then $\nu$ is not rigid.
If $\mu\simeq\nu$, then for some $t\ne 0$, $\mu_t\in V$ and $\mu\not\simeq\mu$ so that
$\mu_t\not\simeq\nu$ and therefore $\nu$ is not rigid.
\end{proof}

For the ease of completeness we briefly address the problem in dimensions $\le 8$.
See Remark \ref{rmk:78}.

\subsection{Dimension 8}

Complex filiform Lie algebras of dimension 8 are classified \cite{AG}. 
Based on it, in \cite{GT} it is shown that non of them are rigid in $\L^8$.
This follows by constructing a solvable non-nilpotent deformation of each of them.

However, it remains the question whether they are or not rigid in $\F^8$.
The answer is no, they are not rigid. 

The same construction we used for dimensions 9, 10 and 11, applies successfully in dimension 8.
In this case, for a given $\mu\in\F^8$, choose $D=D^3$ and consider the corresponding deformation
$\mu_t$ as in \eqref{eqn:mut}. 
It turns out that $\mu\simeq\mu_t$ only if $t=0$.
This follows much easier than in the cases we treated,
even though Proposition \ref{prop:g} does not hold for $n=8$.

\subsection{Dimension 7}

The situation in dimension 7 is very similar to that in dimension 8.
On the one hand all of them admit solvable non-nilpotent deformations.
On the other hand, our method (choosing $D=D^3$) produces non-trivial
filiform deformations of all of them.

\vskip2cm

\noindent{\bf Acknowledgements.}
This paper is part of the PhD.\ thesis of the second author. 
She thanks CONICET for the Ph.D.\ fellowship awarded that made this possible.


\end{document}